\documentclass[11pt]{article}

\usepackage{amsthm}
\usepackage{amsmath}
\usepackage{amssymb}
\usepackage{amsmath, amsthm, bbm}
\usepackage{framed, fullpage}

\usepackage[backref,colorlinks,citecolor=blue,bookmarks=true]{hyperref}

\newtheorem{theo}{Theorem}

\newtheorem{lemma}{Lemma}[section]

\newtheorem{claim}[lemma]{Claim}

\newcommand{\E}{\mathbb{E}}

\newcommand{\size}[1]{\left\lvert #1 \right\rvert}
\newcommand{\norm}[1]{\left\lVert #1 \right\rVert}
\newcommand{\floor}[1]{\left\lfloor{#1}\right\rfloor}
\newcommand{\ceil}[1]{\left\lceil #1 \right\rceil}

\DeclareMathOperator{\twr}{twr}

\renewcommand{\b}{\beta}
\renewcommand{\d}{\delta}

\newcommand{\e}{\epsilon}
\newcommand{\z}{\zeta}
\renewcommand{\l}{\lambda}
\renewcommand{\ss}{\subseteq}
\renewcommand{\ss}[1]{\subseteq_{#1}}
\newcommand{\ssb}{\ss{\b}}

\newcommand{\X}{\mathcal{X}}
\newcommand{\Z}{\mathcal{Z}}

\date{}

\title{A Short Proof of Gowers' Lower Bound for the Regularity Lemma}
\author{
Guy Moshkovitz\thanks{School of Mathematics, Tel-Aviv University, Tel-Aviv, Israel 69978.  Email: {\tt guymosko@tau.ac.il}. Supported in part by ISF grant 224/11.}
\and Asaf Shapira\thanks{School of Mathematics, Tel-Aviv University, Tel-Aviv, Israel 69978. Email: {\tt asafico@tau.ac.il}. Supported in part by ISF Grant 224/11 and a Marie-Curie CIG Grant 303320.}
}
\begin{document}

\maketitle
\begin{abstract}

A celebrated result of Gowers states that for every $\epsilon >0$ there is a graph $G$ so that every
$\epsilon$-regular partition of $G$ (in the sense of Szemer\'edi's regularity lemma) has order given by a tower of exponents of height polynomial in $1/\epsilon$.
In this note we give a new proof of this result that uses a construction and proof of correctness that are significantly simpler
and shorter.

\end{abstract}

\section{Introduction}

Szemer\'edi's regularity lemma asserts that every graph can be partitioned into a bounded number of vertex sets $Z_1,\ldots,Z_k$, so that the graphs between almost
all pairs $(Z_i,Z_j)$ behave ``randomly''. More precisely, given two vertex sets in a graph $G$ let $d_G(A,B)=e(A,B)/|A||B|$ where $e(A,B)$ is the number of
edges in $G$ with one vertex in $A$ and the other in $B$. We say that the pair $(A,B)$ is $\epsilon$-regular if $|d_G(A,B)-d_G(A',B')| \leq \epsilon$ for all $A' \subseteq A$ and $B' \subseteq B$
satisfying $|A'|\geq \epsilon|A|$ and $|B'|\geq \epsilon |B|$. A partition $\Z=\{Z_1,\ldots,Z_k\}$ of the vertex set of a graph is called an {\em equipartition} if all
the sizes of the sets $Z_i$ differ by at most $1$. The {\em order} of an equipartition $\Z$, denoted $|\Z|$, is the number of sets in it ($k$ above).
An equipartition $\Z$ is $\epsilon$-regular if all but $\epsilon k^2$ of the pairs $(Z_i,Z_j)$ are $\epsilon$-regular\footnote{Note that we do not require the sets $Z_i,Z_j$ in the pairs $(Z_i,Z_j)$ to be distinct. Therefore, we do not need a lower bound
on the order of the partition in Theorem \ref{theo:sz}.}.
Szemer\'edi's regularity lemma then states the following.

\begin{theo}[Szemer\'edi \cite{Sz}]\label{theo:sz} For every $\epsilon >0$ there is $M=M(\epsilon)$ so that every graph has an $\epsilon$-regular
equipartition of order at most $M$.
\end{theo}

Despite its apparent simple statement (and proof) the regularity lemma has become one of the most widely used tools in extremal graph theory, as well as in many other fields (see \cite{KSim} for a survey). Unfortunately, the proof in \cite{Sz} only showed that $M(\epsilon) \leq \twr(O(1/\epsilon^5))$
where $\twr(x)$ is a tower of exponents of height $x$. Hence, the numerous applications of the lemma were all of asymptotic nature and supplied very weak effective bounds.

For a long time it was not known whether the tower-type bound for $M(\epsilon)$ was unavoidable until Gowers proved \cite{Gowers} that (surprisingly) this is indeed the case. Gowers' paper contained two proofs. The ``first'' proof used a simple construction with a short proof of correctness, but it only showed that
$M(\epsilon) \geq \twr(c\log (1/\epsilon))$. The ``second'' proof established the much stronger bound $M(\epsilon) \geq \twr(1/\epsilon^c)$ thus showing that
$M(\epsilon)$ indeed grows as a tower of exponents of height polynomial in $1/\epsilon$ (\cite{Gowers} obtains $c=1/16$).
However, the second proof of the stronger bound used a far more complicated construction with a significantly more involved proof of correctness, and was dubbed a tour-de-force in the laudatio to Gowers' Fields medal \cite{B}. Conlon and Fox \cite{CF} gave another proof of the fact that $M(\epsilon) \geq \twr(1/\epsilon^c)$ (with $c=1$), but their proof was equally involved.

While the proof of Gowers' first construction used an inductive approach, in the concluding remarks to his paper \cite{Gowers} he explained that ``the proof for the second construction is so much more complicated than the proof for the first'' since one cannot use a similar inductive approach in the second construction. Our main contribution here is a new proof that $M(\epsilon) \geq \twr(1/\epsilon^c)$ (we obtain $c=1/6$).
At a high level, our proof is almost identical to the first proof in \cite{Gowers} using a very similar inductive approach. However, the proofs differ is several subtle aspects which make it possible to execute the inductive argument $1/\epsilon^c$ times and not only $\log (1/\epsilon)$ times as in the first proof of \cite{Gowers}.
We finally note that the second construction in \cite{Gowers} as well as the one in \cite{CF} prove lower bounds for weaker versions of the regularity lemma. It would be interesting to see if one could use the ideas in our new proof to give simple proofs of comparable lower bounds for weaker versions of the regularity lemma.

Let us say that an equipartition $\Z=\{Z_1,\ldots,Z_k\}$ is $\epsilon$-{\em nice} if for every $Z \in \Z$ all but $\epsilon k$ of the sets $Z' \in \Z$ are
such that $(Z,Z')$ is $\epsilon$-regular. Let $M'(\epsilon)$ be so that every graph has an $\epsilon$-nice equipartition of order at most $M'(\epsilon)$.
It is a well-known (and easy) observation that $M'(\epsilon) \leq M(\epsilon^{3})$. Hence, to prove that $M(\epsilon) \geq \twr(1/\epsilon^{1/6})$ it would suffice to prove the following.

\begin{theo}\label{theo:main}
There is a constant $c>0$ so that $M'(\epsilon) \geq \twr(c/\e^{1/2})$ for every $0<\epsilon<c$.
\end{theo}

\section{Proof of Theorem \ref{theo:main}}

\subsection{Preliminary lemmas}

Suppose $G$ is a weighted complete graph, where each edge $(x,y)$ is assigned a weight $d_G(x,y) \in [0,1]$.
Given two vertex sets $A,B$ in $G$ define the weighted density between $A,B$ by $d_G(A,B)=\sum_{x \in A, y \in B}d_G(x,y)/|A||B|$.
The following claim follows immediately from Chernoff's inequality.

\begin{claim}\label{lemma:weightedGraph}
Let $\zeta > 0$. Suppose $G$ is a weighted complete graph on $n$ vertices with weights in $[0,1]$, and $G'$ is a random graph, where each edge $(x,y)$ is chosen independently to be included in $G'$
with probability $d_G(x,y)$. Then with probability at least $1/2$ we have $|d_{G'}(A,B)-d_G(A,B)| \leq \zeta$ for all sets $A,B$ of size at least $20\zeta^{-2}\log(n)$.
\qed
\end{claim}

A pair of vertex sets $A,B$ in a weighted graph $G$ are $\e$-regular if $|d_G(A,B)-d_G(A',B')| \le \e$ for
all $A' \subseteq A$ and $B' \subseteq B$ satisfying $|A'|\geq \epsilon|A|$ and $|B'|\geq \epsilon |B|$.
Clearly if one can construct a large weighted graph $G$ with
the property that every $\epsilon$-nice equipartition of $G$ is of size $\twr(1/\epsilon^c)$, then an application
of Claim \ref{lemma:weightedGraph} will then give a ``genuine'' graph $G'$ with the same property.
Hence, we will prove our lower bound on $M'(\epsilon)$ with respect to weighted graphs.

If $M$ is an even integer, then a sequence $(A_i,B_i)^m_{i=1}$ of $m$ bipartitions of $[M]$ is called \emph{$c$-balanced}
if for every $i$ we have $|A_i|=|B_i|=M/2$ and for every distinct $t,t' \in [M]$ there are at most $(\frac12+c)m$ values $i$ for which $t,t'$ lie in the same part of $(A_i,B_i)$.

\begin{lemma}\label{lemma:pairwise}
For every $m\geq 1$ and $M=2^{\ceil{m/512}}$ there exists a sequence of $m$ bipartitions of $[M]$ that is $\frac{1}{16}$-balanced.
\end{lemma}
\begin{proof} If $m\leq 512$ then $M=2$ and we can just take $m$ identical copies of the partition $A=\{1\}$ and $B=\{2\}$.
Suppose now that $m>512$. We choose, uniformly at random, $m$ bipartitions of $[M]$ into two sets of equal size, with the choices being mutually independent.
Fix $t\neq t' \in[M]$. The probability that $t$ and $t'$ are in the same part of a given bipartition is $2\binom{M-2}{M/2-2}/\binom{M}{M/2} \le 1/2$.
By Chernoff's inequality, the union bound, and the fact that $m>512$, the probability that some pair $t \neq t' \in [M]$ belongs to the same part for more than $9m/16$ of the bipartitions is at most ${M \choose 2} \exp(-2m/256) < 1$, so the required sequence of partitions exists.
\end{proof}

\begin{lemma}\label{lemma:biased}
If $(A_i,B_i)^m_{i=1}$ is a sequence of bipartitions of $[M]$ that is $\frac{1}{16}$-balanced, then for every $\l=(\l_1,\ldots,\l_M)$ with $\l_t \ge 0$, $\norm{\l}_{1}=1$, and $\norm{\l}_{\infty} \le 1-8\z$, at least $m/6$ of the bipartitions $(A_i,B_i)$ satisfy $\min(\sum_{t\in A_i} \l_t,\sum_{t\in B_i} \l_t) \ge \z$.
\end{lemma}
\begin{proof} Suppose $(A_i,B_i)^m_{i=1}$ is a sequence of partitions of $[M]$ that is $\frac14$-balanced (and not necessarily $\frac{1}{16}$-balanced).
We first show that in this case we can find one bipartition $(A_i,B_i)$ satisfying $\min(\sum_{t\in A_i} \l_t, \sum_{t\in B_i} \l_t) \geq \z$.
Choose one of the partitions $(A_i,B_i)$ in the sequence uniformly at random, and let $Y_t$ be the random variable satisfying $Y_t=1$ if $t\in A_i$, and $Y_t=-1$ if $t\in B_i$.
Clearly, $\E[Y_t^2]=1$ and as the sequence is $\frac14$-balanced, we have $\E[Y_tY_{t'}]\le 1/2$ for every $t\neq t'\in[M]$.
Let $Y=\sum_{t=1}^M \l_t Y_t$. Then
$$\E[Y^2] \leq \sum_t \l_t^2 + \frac12\sum_{t\neq t'} \l_t\l_{t'} =
\frac12\sum_t \l_t^2+\frac12(\sum_t \l_t)^2=\frac12+\frac12\sum_t \l_t^2 \leq 1-4\zeta \;,$$
where in the last inequality we used the fact that $\sum_t\l_t^2 \leq \norm{\l}_{\infty} \cdot \norm{\l}_{1} \leq 1-8\zeta$.
We conclude that $\E[|Y|]\leq 1-2\zeta$, implying that there exists an $i$ for which the bipartition $(A_i,B_i)$ satisfies
$\big\lvert \sum_{t\in A_i} \l_t - \sum_{t\in B_i} \l_t \big\rvert \le 1-2\z$.
Since $\big\lvert \sum_{t\in A_i} \l_t - \sum_{t\in B_i} \l_t \big\rvert=2\bigg\lvert \sum_{t\in A_i} \l_t - \frac12 \bigg\rvert$
this means that $\z \le \sum_{t\in A_i} \l_t \le 1-\z$, implying that
$\min(\sum_{t\in A_i} \l_t, \sum_{t\in B_i} \l_t) \ge \z$, as desired.

Suppose now that our sequence of $m$ bipartitions is $\frac{1}{16}$-balanced. We repeatedly apply the argument from the previous paragraph, where in each step we ``pull out'' a bipartition satisfying $\min(\sum_{t\in A_i} \l_t, \sum_{t\in B_i} \l_t) \ge \z$.
By the claim in the previous paragraph we can do this as long as the remaining set of bipartitions is $\frac14$-balanced.
We claim that as long as we have removed less than $m/6$ of the bipartitions, the remaining sequence is still $\frac14$-balanced.
Indeed, since the original sequence was $\frac{1}{16}$-balanced, if we remove at most $m/6$ bipartitions, then for each pair $t\neq t' \in [M]$
the fraction of bipartitions in which $t,t'$ belong to the same part is at most
$(9m/16)/(5m/6)\leq 3/4$.
\end{proof}

\subsection{The construction}\label{subsec:Construction}

We now describe the weighted graph $G=(V,E)$ on $n$ vertices which, as we will shortly prove, has no small $\e$-nice equipartition.
Henceforth, set $\d = 30\e^{1/2}$, $s=\floor{1/\delta}$, $\phi(m)=2^{\ceil{m/512}}$ and assume $n$ is large enough as a function of $\epsilon$,
and that $0< \epsilon < c$ for some small enough $c$.

Let $\X_0 , \X_1 , \ldots , \X_s$ be a sequence of $s+1$ equipartitions of $V$ each refining the previous one, where $\X_{0}$ is the trivial partition with $\size{\X_0}=1$, such that every part of $\X_{r-1}$ is subdivided into $\phi(\size{X_{r-1}})$ parts in $\X_r$.
Note that $\size{\X_r} = \size{\X_{r-1}}\cdot\phi(\size{\X_{r-1}})$, implying that $\size{\X_r} = \twr(\Omega(r))$.
For each $1\le r\le s$ we now define a weighted graph $G_r$ using the partitions $\X_{r-1}$ and $\X_r$.
For convenience, write $\X_{r-1} =\{X_1,\ldots,X_m\}$ and $\X_r = \{X_{i,t}\}_{i=1,t=1}^{m,M}$
so that the sets $X_{i,1},\ldots,X_{i,M}$ form a partition of $X_i$.
Let $(A'_{j},B'_{j})_{j=1}^m$ be a sequence of $m$ bipartitions of $[M]$ that is $\frac{1}{16}$-balanced, as in Lemma~\ref{lemma:pairwise} (we can choose these bipartitions since $M=\phi(m)=2^{\ceil{m/512}}$).
For each $1 \leq i \leq m$, we assign to $X_i$ a sequence of $m$ bipartitions $(A_{i,j},B_{i,j})_{j=1}^m$ of its vertices by letting $A_{i,j}:=\bigcup_{t\in A'_{j}} X_{i,t}$ and $B_{i,j}:=\bigcup_{t\in B'_{j}} X_{i,t}$ (i.e., we think of each bipartition $(A'_{j},B'_{j})$ as a bipartition of the collection of sets $X_{i,1}\ldots,X_{i,M}$).
Now, for every $u\in X_i$ and $v\in X_j$, the edge $(u,v)$ has a positive weight $\d$ in $G_r$ if and only if $u\in A_{i,j}$ and $v\in A_{j,i}$ or $u\in B_{i,j}$ and $v\in B_{j,i}$. Notice we allow $i=j$ in the above; moreover, we allow (for convenience) self loops. As an example, notice $\size{\X_0}=1$ and $\size{\X_1}=2$,
so $G_1$ is just a vertex-disjoint union of two cliques, each on exactly half the vertices, whose edges are all of weight $\d$.
Finally, define $G=G_1+G_2+\cdots+G_s$, meaning that the final weight assigned to each edge is the sum of the weights assigned to this edge over all graphs $G_1,\ldots,G_s$.
This is well defined as the weight of each edge is at most $s\d \le 1$.

We now state an important observation regarding the graph $G$. 
Fix an integer $1 \leq r \leq s$, a set $X_i \in \X_{r-1}$, a vertex $v \in X_i$ and $1 \leq j \leq m$. Since in the construction above the bipartition $X_j = A_{j,i} \cup B_{j,i}$ satisfies $\size{A_{j,i}}=\size{B_{j,i}}=|X_j|/2$, we see that $d_{G_{r}}(v,X_j)=\delta/2$.
Since each set $X \in \X_{r}$ is a disjoint unions of sets $X' \in \X_{r+1}$ etc. and since the partitions $\X_{1},\ldots,\X_{s}$ are equipartitions, we get that for
every $1 \leq r \leq s$, vertex $v\in V$, and $X \in \X_{r}$ that $d_{G_{r+1}+\cdots+G_s}(v,X)=\frac12\d(s-r)$.
Finally, since the sets $A_{j,i}$ are disjoint unions of sets $X \in \X_{r}$ we get that for every set $A_{j,i}$ and for {\em any} other set of vertices $Z$, we have
\begin{equation}\label{eq:density}
d_{G_{r+1}+\cdots+G_s}(Z,A_{j,i})=\frac12\d(s-r) \;.
\end{equation}

\subsection{Proof of correctness}

We write $A\ssb B$ to denote the fact that $|A\cap B| \geq (1-\beta)|A|$.
We say that a partition $\Z$ is a \emph{$\b$-refinement} of a partition $\X$ if for every $Z\in\Z$ there is an $X\in\X$ such that $Z\ssb X$.
Note that if $\Z$ is a $\b$-refinement of $\X$ with $\b<1/2$ then, in particular, each $Z\in\Z$ satisfies $Z\ssb X$ for a unique $X\in \X$.
In what follows, we only consider $\b$-refinements with $\b<1/2$.
The heart of the proof of Theorem~\ref{theo:main} is the following lemma, in which $G$, $s$ and $\delta$ are those defined above.

\begin{lemma}\label{lemma:main}
Suppose $\b \le \d/60 < 1/2$, and $1 \leq r \leq s$.
If $\Z$ is an $\e$-nice equipartition of $G$ that $\b$-refines $\X_{r-1}$ then it $(\b+8\e)$-refines $\X_r$.
\end{lemma}

We first deduce Theorem \ref{theo:main} from Lemma \ref{lemma:main} and then prove the lemma.

\begin{proof}[Proof of Theorem \ref{theo:main}]
Let $\Z$ be an $\e$-nice equipartition of the weighted graph $G$. Since $\Z$ is a $0$-refinement of $\X_0=V$, it follows from repeated applications of Lemma~\ref{lemma:main} that $\Z$ is an $r \cdot 8\e$-refinement of $\X_r$, for every $r \le \d/(60\cdot 8 \e)$, and in particular, for $r=s$.
We thus get that $\Z$ is a $\b$-refinement of $\X_s$ with $\b=s \cdot 8\e \leq \epsilon^{1/2} \ll 1/2$, which implies $\size{\Z}\ge\size{\X_s}/2$.
As mentioned earlier, $\size{\X_s} = \twr(\Omega(s)) \geq \twr(c/\epsilon^{1/2})$, thus proving the desired lower bound on $\size{\Z}$.
Finally, as noted earlier, it follows from Claim~\ref{lemma:weightedGraph} that there exists a (non-weighted) graph $G'$ satisfying the same conclusion, thus completing the proof.
\end{proof}

\begin{proof}[Proof of Lemma~\ref{lemma:main}]
Write $\X_{r-1}=\{X_1,\ldots,X_m\}$ and $\X_r=\{X_{i,t}\}_{i=1,t=1}^{m,M}$.
Suppose to the contrary that there exists $Z_0\in\Z$ with $Z_0\ssb X_i$ such that $Z_0\not\ss{\b+8\e} X_{i,t}$ for every $1 \leq t \leq M$.
Write $k=\size{\Z}$.
We will show that there are at least $\e k$ 
sets $Z\in\Z$ such that $(Z_0,Z)$ is $\e$-irregular.

Call a vertex $v\in X_j$ \emph{useful} if the unique $Z\in\Z$ containing $v$ satisfies $Z\ssb X_j$.
Call a set $X_j$ \emph{useful} if it contains at least $(1-12\b)\size{X_j}$ useful vertices, and moreover, the bipartition $X_i = A_{i,j} \cup B_{i,j}$ satisfies $\min(\size{Z_0\cap A_{i,j}},\size{Z_0\cap B_{i,j}}) \ge \e\size{Z_0}$.
We now show that there are at least $m/12$ useful sets $X_j\in\X_{r-1}$.
First, note that as $\Z$ is a $\b$-refinement of $\X_{r-1}$, at most $\b n$ of all vertices are non-useful.
Hence by averaging, there are at most $m/12$ sets $X_j\in \X_{r-1}$ containing more than $12\b\size{X_j}$ non-useful vertices.
Second, for each $1\le t\le M$ set $\l_t=\size{Z_0\cap X_{i,t}}/\size{Z_0\cap X_i}$. Denoting $\l=(\l_1,\ldots,\l_M)$, we have $\norm{\l}_1=1$ and, as $Z_0\not\ss{\b+8\e} X_{i,t}$ for all $t$, we have
$\norm{\l}_{\infty} < 1-8\cdot \e/(1-\b)$.
Since the sequence of bipartitions $(A_{i,j},B_{i,j})_{j=1}^m$ is (by construction) $\frac{1}{16}$-balanced, we get from Lemma~\ref{lemma:biased} (with $\z=\e/(1-\b)$) that there are at least $m/6$ values $j$ for which the  bipartition $(A_{i,j},B_{i,j})$ is such that both $\sum_{t\in A_{i,j}} \size{Z_0\cap X_{i,t}}/\size{Z_0\cap X_i}$ and $\sum_{t\in B_{i,j}} \size{Z_0\cap X_{i,t}}/\size{Z_0\cap X_i}$ are at least $\e/(1-\b)$.
For each such bipartition we have
$\min(\size{Z_0\cap A_{i,j}},\size{Z_0\cap B_{i,j}}) \ge \e\size{Z_0}$.
We conclude that there are at least $m/6-m/12=m/12$ values $j$ for which $X_j$ is useful.

Fix a useful set $X_j$. Let $\Z_j=\{Z\in\Z : Z\ssb X_j\}$. We now show that there are at least $12\epsilon k/m$ sets $Z \in \Z_j$ so that
$(Z_0,Z)$ is $\epsilon$-irregular. Together with the fact that there are at least $m/12$ useful sets $X_j$ we will thus get the required $\epsilon k$ (distinct) sets
$Z$ for which $(Z_0,Z)$ is $\epsilon$-irregular.
So suppose to the contrary that $\Z_j$ contains less than $12\epsilon k/m$ sets that together with $Z_0$ form an $\epsilon$-irregular pair.
Set $F=G_{r+1}+\cdots+G_s$, $Z^1=Z_0\cap A_{i,j}$ and $Z^2=Z_0\cap B_{i,j}$.
Since $X_j$ is useful, we have $\min(|Z^1|,|Z^2|)\geq \epsilon |Z_0|$.
Let $A \subseteq A_{j,i}$ be the set of vertices $x$ satisfying one of the following; $(i)$ $x$ it is not useful, $(ii)$ $x$ belongs to a set $Z \in \Z_j$ so that $(Z_0,Z)$ is $\epsilon$-irregular, $(iii)$ $x$ belongs to a set $Z \in \Z_j$ so that $(Z_0,Z)$ is $\epsilon$-regular and $d_F(x,Z^2)< d_F(x,Z^1)+\frac34\d$.
Since $X_j$ is useful and $|A_{j,i}|=|X_j|/2$ we infer that $A_{j,i}$ has at most $24\b\size{A_{j,i}}$ vertices satisfying~$(i)$. Our assumption on the number of $\e$-irregular pairs $(Z_0,Z)$ with $Z \in \Z_j$ implies that there at most $24\e\size{A_{j,i}}$ vertices satisfying~$(ii)$. Suppose $Z \in \Z_j$ and $(Z_0,Z)$ is $\e$-regular.
Further suppose that $Z$ contains a subset $Z' \subseteq Z$ of at least $\e\size{Z}$ vertices all of which satisfy~$(iii)$.
Since $Z'\subseteq A_{j,i}$, $Z^1\subseteq A_{i,j}$ and $Z^2\subseteq B_{i,j}$ we have by construction that $d_{G_r}(Z',Z^1)=\d$ and $d_{G_r}(Z',Z^2)=0$. Moreover, notice that\footnote{This is true for any $Z',Z_1,Z_2$ for which there are $X_i,X_j \in \X_{r-1}$ satisfying  $Z^1,Z^2 \subseteq X_i$ and $Z' \subseteq X_j$.}  $d_{G_{\ell}}(Z',Z^1)=d_{G_{\ell}}(Z',Z^2)$ for every $1 \leq \ell \leq r-1$.
Therefore, in this case we would get that
\begin{eqnarray*}
d_G(Z',Z^1)-d_G(Z',Z^2)&=&d_{G_r}(Z',Z^1)-d_{G_r}(Z',Z^2)+d_{F}(Z',Z^1)-d_{F}(Z',Z^2)\\
&>& \delta-\frac34\delta > 2\epsilon\;,
\end{eqnarray*}
contradicting the fact that $(Z_0,Z)$ is $\e$-regular. We thus get that $A_{j,i}$ contains at most $4\e\size{A_{j,i}}$ vertices satisfying~$(iii)$,
implying that altogether $\size{A} \leq (24\b+28\e)\size{A_{j,i}} \leq \frac12\d\size{A_{j,i}}$.
Note that if $x \not\in A$ then
$d_F(x,Z^2)-d_F(x,Z^1) \ge \frac34\delta$, hence we can conclude that
\begin{eqnarray*}
d_F(A_{j,i},Z^2)-d_F(A_{j,i},Z^1)&=&\frac{1}{|A_{j,i}|}\left(\sum_{x \not \in A}d_F(x,Z^2)-d_F(x,Z^1) + \sum_{x \in A}d_F(x,Z^2)-d_F(x,Z^1)\right)\\
&\ge& (1-\d/2)\frac34\delta - \d/2 > 0\;.
\end{eqnarray*}
But this contradicts (\ref{eq:density}), thus completing the proof.
\end{proof}


\begin{thebibliography}{99}

\bibitem{B}
B. Bollob\'{a}s, The work of William Timothy Gowers, Proc. of the ICM, Vol. I (Berlin, 1998). Doc. Math., Extra Vol. I, 1998, 109-118 (electronic).

\bibitem{CF}
D. Conlon and J. Fox, Bounds for graph regularity and removal lemmas, GAFA 22 (2012), 1191-1256.

\bibitem{Gowers}
T. Gowers, Lower bounds of tower type for Szemer\'edi's uniformity
lemma, GAFA 7 (1997), 322-337.

\bibitem{KSim}
J.~Koml\'os and M.~Simonovits, Szemer\'edi's Regularity Lemma and
its applications in graph theory. In: {\em Combinatorics, Paul
Erd\"os is Eighty}, Vol II (D.~Mikl\'os, V.~T.~S\'os, T.~Sz\"onyi
eds.), J\'anos Bolyai Math.~Soc., Budapest (1996), 295--352.


\bibitem{Sz} E.~Szemer\'edi,
Regular partitions of graphs, In: {\em Proc.\ Colloque Inter.\
CNRS} (J.~C.~Bermond, J.~C.~Fournier, M.~Las~Vergnas and
D.~Sotteau, eds.), 1978, 399--401.


\end{thebibliography}
\end{document}